\documentclass{amsart}
\usepackage{amssymb, latexsym}

\DeclareMathOperator{\M}{\mathcal{M}}
\DeclareMathOperator{\N}{\mathcal{N}}

\DeclareMathOperator{\rank}{rank}

\DeclareMathOperator{\A}{\mathcal{A}}

\theoremstyle{plain}
\newtheorem{theorem}{Theorem}
\newtheorem{corollary}{Corollary}
\newtheorem{lemma}{Lemma}

\theoremstyle{definition}
\newtheorem{definition}{Definition}
\begin{document}
\title[Determinantal polynomials]
{Properties of determinantal polynomials of subspaces 
 of matrices over a finite field}

\author[R. Gow]{Rod Gow}
\address{School of Mathematics and Statistics\\
University College Dublin\\
 Ireland}
\email{rod.gow@ucd.ie}


\begin{abstract} 
Let $K$ be a field and let $M_n(K)$ denote the space of $n\times n$ matrices with entries in $K$. Let $\M$ be a subspace of $M_n(K)$
of dimension $d$ with the property that there are elements in $\M$ with non-zero determinant. Given a basis of $\M$, we define the determinantal polynomial $P_{\M}$ of $\M$ with respect to the basis. It is a homogeneous polynomial of degree $n$ in $d$ indeterminates
that gives the determinant of any element of $\M$ by evaluation in $K^d$. 

This paper investigates the interrelationship of $\M$ and $P_{\M}$. We largely confine ourselves to finite fields $K$, where we can obtain
useful information by applying the Lang-Weil theorem on the number of zeros of absolutely irreducible polynomials. A combination
of Chevalley's theorem on the zeros of polynomials in several variables and the Lang-Weil theorem leads to theorems about the characteristic
polynomials of elements of $\M$ when $n$  is a prime. We also draw attention to cases when the elements of $\M$ with non-zero determinant
are a proper subspace, and provide non-trivial examples of this  phenomenon.
\end{abstract}
\maketitle
\section{Introduction}

\noindent  Let $K$ be a field and let $M_n(K)$ denote the $K$-vector space of $n\times n$ matrices with entries in $K$. We are interested in studying subspaces of $M_n(K)$, especially those which contain elements with non-zero determinant. With this aim in mind, we make the following definition. 

\begin{definition}
Let $\M$ be a $d$-dimensional subspace of $M_n(K)$, where $d\geq 1$. Let $M_1$, \dots, $M_d$ be a basis of $\M$ and let $x_1$, \dots, $x_d$ be independent indeterminates over $K$.
We set
\[
P_{\M}(x_1, \ldots, x_d)=\det (x_1M_1+\cdots +x_dM_d)
\]
and we call $P_{\M}$ the \emph{determinantal polynomial} of $\M$ (with respect to the given
basis).
\end{definition}

Clearly, different choices of basis of $\M$ will lead to different determinantal polynomials, and so we cannot speak of a unique such polynomial. However, suppose that 
$B_1$, \dots, $B_d$ is another basis of $\M$ and we set
\[
Q_{\M}(y_1, \ldots, y_d)=\det(y_1B_1+ \cdots +y_dB_d).
\]
Then, since each basis is a $K$-linear combination of the other, $P_{\M}$ and $Q_{\M}$ are obtained from each other by a linear change of variables. Thus, properties such reducibility
or irreducibility of the determinantal polynomials are uniquely determined independently of the basis, and it is properties of this kind which concern us in this paper. 

We also briefly mention the concept of \emph{equivalence} of subspaces. Let $\M$ be a $d$-dimensional subspace of $M_n(K)$ and let $C$ and $D$ be invertible elements of $M_n(K)$.
The set of all elements of the form $CMD$, as $M$ runs over $\M$, is another subspace, $\N$, say, of $M_n(K)$ of dimension $d$ and we say that it is equivalent to $\M$. If
$M_1$, \dots, $M_d$ is a basis of $\M$, then $B_1=CM_1D$, \dots, $B_d=CM_dD$ is a basis of $\N$, and the multiplicative property of determinants shows that the determinantal polynomial
 of $\N$ with respect to this basis is a non-zero scalar multiple of that of $\M$. Thus the determinantal polynomials of equivalent subspaces are either both reducible with 
 irreducible factors of the same degrees or both irreducible or both absolutely irreducible.
 
 The notion of equivalent subspaces is important when one studies such subjects as semifields, which may be realized as subspaces of matrices where all non-zero elements have non-zero determinant. 
 
 We mention another simple property of determinantal polynomials which we shall use, frequently without comment. Suppose that $\M_1$ is an $r$-dimensional subspace of $\M$
and let $M_1$, \dots, $M_r$ be a basis of $\M_1$. We extend this basis to a basis
$M_1$, \dots, $M_d$ of $\M$ and define the determinantal polynomial of $\M$ with respect to this basis. Then it is easy to see that
\[
P_{\M}(x_1, \ldots, x_r, 0, \ldots, 0)=P_{\M_1}(x_1, \ldots, x_r).
\]
Thus, for example, it follows that if the determinantal polynomial of $\M_1$ is irreducible,
that of $\M$ is also irreducible. 

We must draw attention to another point that will not concern us much but is nonetheless of interest. Let $\M$ be a $d$-dimensional subspace of $M_n(K)$. It is obvious that
if $\M$ contains elements of non-zero determinant, its determinantal polynomial is non-zero.
The converse is not true, but examples of this unusual behaviour only occur when $K$ is finite and $|K|$ is small compared with $n$, as we now show.

\begin{theorem} \label{field_size}
Let $\M$ be a subspace of $M_n(\mathbb{F}_q)$ of dimension $d>0$, all of whose elements have zero determinant. Suppose that $P_{\M}$ is not the zero polynomial. Then $q<n$.
\end{theorem} 

\begin{proof}
We note that $P_{\M}$ is homogeneous of degree $n$. It follows from Theorem 6.15 of \cite{LN}
that $P_{\M}$ has at most $n(q^{d-1}-1)$ non-trivial zeros on $\mathbb{F}_q^d$. Since we are assuming that all elements of $\M$ have zero determinant, $P_{\M}$ vanishes
on $\mathbb{F}_q^d$. Thus we have the inequality
\[
q^{d}-1\leq n(q^{d-1}-1)
\]
and this implies that $q<n$.
\end{proof}

\noindent{\bf Example 1} Let $\M$ be the two-dimensional subspace of $M_3(\mathbb{F}_2)$ with basis vectors
\[
\left(
\begin{array}
{ccc}
        1&0&0\\
        0&1&0\\
        0&0&0   
\end{array}
\right), \qquad \left(
\begin{array}
{ccc}
        0&0&0\\
        0&1&0\\
        0&0&1 
\end{array}
\right).
\]
Then we find that $P_{\M}(x_1,x_2)=x_1x_2(x_1+x_2)$ and this polynomial is zero on
$\mathbb{F}_2^2$.
\bigskip

\noindent{\bf Example 2} Let $\N$ be the two-dimensional subspace of $M_4(\mathbb{F}_3)$ with basis vectors
\[
\left(
\begin{array}
{cccc}
        1&0&0&0\\
        0&1&0&0\\
        0&0&1&0\\
        0&0&0&0  
\end{array}
\right), \quad \left(
\begin{array}
{cccc}
        2&0&0&0\\
        0&1&0&0\\
        0&0&0&0\\
        0&0&0&1 
\end{array}
\right).
\]
Then we find that $P_{\N}(x_1,x_2)=x_1x_2(x_1+x_2)(x_1-x_2)$ and this polynomial vanishes on
$\mathbb{F}_3^2$.

We note the following property of the subspaces considered in the examples above.
Let $\M$ be any subspace of $M_n(\mathbb{F}_q)$ such that all elements of $\M$ have determinant 0 and yet $P_{\M}$ is non-zero. Then any subspace $\M'$ of $M_n(\mathbb{F}_q)$
that contains $\M$ and whose elements all have determinant 0 also has non-zero determinantal polynomial. This enables us to construct subspaces of larger dimension with the vanishing determinant/non-vanishing determinantal polynomial property. For example, we can embed
the subspace of Example 2 into an 8-dimensional subspace with non-zero determinantal polynomial, whose elements all have zero
determinant.

We begin a more systematic investigation of determinantal polynomials and their use in investigating such things as characteristic polynomials. We make good use of the following lemma, for which we claim no originality. As we could not find a proof in a textbook, we provide the following argument.

\begin{lemma} \label{homogeneous}
Let $f\in K[x_1, \ldots,x_d]$ be a non-constant homogeneous polynomial of degree $n\geq 1$.
Suppose that $f=gh$ is a factorization of $f$ into non-constant polynomials $g$ and $h$ in
$K[x_1, \ldots,x_d]$. Then $g$ and $h$ are both homogeneous.
\end{lemma}

\begin{proof}
Let $r$ be the degree of $g$ and $s$ the degree of $h$, where $r+s=n$. Given a non-negative
integer $t$, let $g^{(t)}$ be defined as the sum of all monomials of degree $t$ that occur
in $g$ with non-zero coefficient (each monomial being multiplied by its coefficient). Thus
\[
g^{(t)}=\sum \lambda(a_1, \ldots, a_d)x_1^{a_1}\cdots x_d^{a_d}, \quad a_1+\cdots+a_d=t.
\]
Clearly, $g^{(t)}$ is homogeneous of degree $t$ and $g$ is the sum of its constituent
homogeneous parts $g^{(t)}$. Let the homogeneous polynomials $h^{(t)}$ be similarly defined for $h$.

We note that a product $g^{(t)}h^{(u)}$ is homogeneous of degree $t+u$ and $f$ is the sum
of such products. Thus, comparing degrees, $f= g^{(r)}h^{(s)}$. We want to show that
$g=g^{(r)}$ and $h=h^{(s)}$, so that $g$ and $h$ are both homogeneous. To this end,
let $t_1=\min t $, $g^{(t)}\neq 0$ and let $u_1=\min u $, $h^{(u)}\neq 0$. The product
$g^{(t_1)}h^{(u_1)}\neq 0$ and it contributes monomials of smallest possible degree $t_1+u_1$ with non-zero coefficient to the product $f= g^{(r)}h^{(s)}$. Unless 
$t_1+u_1=r+s=n$, $f$ cannot be homogeneous of degree $n$. Thus $t_1=r$ and $u_1=s$ and this implies that $g$ and $h$ are both homogeneous.

\end{proof}

\begin{corollary} \label{irreducible}
Let  $f$ be a non-constant homogeneous polynomial in $\mathbb{F}_q[x_1, \ldots, x_d]$. Suppose that $2d\geq 1+\deg f$ and $f$ has no non-trivial zeros in
$\mathbb{F}_q^d$. Then $f$ is an irreducible polynomial.
\end{corollary}

\begin{proof}
Suppose $f$ is reducible and we have a non-trivial factorization $f=gh$ into a product of polynomials $g$ and $h$. Our assumption that $2d\geq 1+\deg f$ implies that at least one of the factors has degree less than $d$. For definiteness, we may then assume that $\deg g<d$. 

Now $g$ is homogeneous by Lemma \ref{homogeneous} and it thus has the trivial zero
$(0, \ldots, 0)$. Since $\deg g<d$, Chevalley's theorem, \cite{LN}, Theorem 6.6, implies that $g$ has a non-trivial zero in $\mathbb{F}_q^d$. But then $f$ also has a non-trivial zero and this contradicts our hypothesis. Thus $f$ is irreducible.
\end{proof}

We note that Chevalley's theorem also implies that $d\leq \deg f$ is a necessary condition
for a polynomial such as $f$ to exist. Furthermore, it is easy to construct examples
of reducible polynomials $f$ in $\mathbb{F}_q[x_1, \ldots, x_d]$ with no non-trivial zeros in $\mathbb{F}_q^d$ if we assume instead that $d\leq \deg f\leq 2d$. 

Throughout this paper, we will employ the following notation. Given a non-zero subspace $\M$ of matrices, $\M^\times$ will denote the subset
of non-zero elements of $\M$.

\begin{corollary} \label{subspace_of_invertibles}
Let $\M$ be a subspace of $M_n(\mathbb{F}_q)$ of dimension $d$. Suppose that each element
of $\M^\times$ has non-zero determinant and $2d>n$. Then the determinantal polynomial $P_{\M}$ is irreducible.
\end{corollary}

\begin{proof}
This follows from Corollary \ref{irreducible}, since $P_{\M}$ has no non-trivial zeros
in $\mathbb{F}_q^d$.
\end{proof}

\begin{corollary} \label{contains_subspace_of_invertibles}
Let $\M$ be a subspace of $M_n(\mathbb{F}_q)$. Suppose that $\M$ contains a subspace $\N$
with the property that each element of $\N^\times$ has non-zero determinant and $2\dim \N>n$. Then $P_{\M}$ is irreducible.
\end{corollary}

\section{The Lang-Weil theorem and its consequences}

\noindent Most of the power of working with the determinantal polynomial $P_{\M}$ of a subspace $\M$ of $M_n(\mathbb{F}_q)$ resides in an appeal to the Lang-Weil theorem, \cite{LW}. Typically, the theorem is used in two ways. It can show that $P_{\M}$ is not absolutely irreducible, although
we may know that the polynomial is irreducible, for example, by an application of  Chevalley's theorem on zeros of polynomials, \cite{LN}, Theorem 6.6. We show for instance that in some cases this enables us to deduce that the elements of $\M$ of determinant zero are a subspace. Alternatively,
knowing that $P_{\M}$ is absolutely irreducible, we obtain a good estimate of the number of elements of $\M$ that have determinant equal to zero.
In both cases, we need to know that $q$ is sufficiently large compared with $n$.

Our purpose in this section is to give an estimate of how large $q$ needs to be in order to be sure that the Lang-Weil theorem applies. 
This analysis is based on an effective form of the Lang-Weil theorem due to Cafure and Matera, \cite{CM}. We will not attempt to obtain an optimal
estimate for the size of $q$, based on currently available effective results, not least as in most cases we are really only interested
in knowing that some theorem is true for all sufficiently large $q$. Of course, it would be of interest to know how much contrary behaviour
can arise
for smaller values of $q$, something that has not been much investigated as far as we are aware, except perhaps in the context of finite semifields.

\begin{theorem} \label{Lang_Weil}
Let $\M$ be a subspace of $M_n(\mathbb{F}_q)$ of dimension $d>0$. Suppose that $P_{\M}$ is absolutely irreducible. Let $N$ be the number
of elements of $\M$ of determinant zero. Then if $q>n^6$ and $n\geq 4$, we have
\[
q^{d-1}\leq 4N\leq 7q^{d-1}.
\]
\end{theorem}

\begin{proof}
Clearly, $N$ is the number of zeros of $P_{\M}$ in $\mathbb{F}_q^d$. Thus, since $P_{\M}$ is absolutely irreducible, and is homogeneous
of degree $n$, the Lang-Weil theorem, \cite{CM}, implies that
\[
|N-q^{d-1}|\leq (n-1)(n-2)q^{d-3/2}+5n^{13/3}q^{d-2}.
\]
We divide by $q^{d-1}$ to obtain
\[
|(N-q^{d-1})/q^{d-1}|\leq (n-1)(n-2)q^{-1/2}+5n^{13/3}q^{-1}.
\]

We now apply the hypothesis that $q>n^6$. Then elementary inequalities imply that
\[
|(N-q^{d-1})/q^{d-1}|\leq n^{-1}+5n^{-5/3}.
\]
We claim that 
\[
10\times 4^{-5/3}<1.
\]
This follows since the cube of left hand side is $1000/1024<1$. Thus, if $n\geq 4$,
\[
|(N-q^{d-1})/q^{d-1}|\leq n^{-1}+5n^{-5/3}<1/4+1/2=3/4.
\]
The desired inequality is an obvious consequence of this estimate.

\end{proof}

We could improve this type of estimate to allow $q>n^5$ if $n$ is a little larger, but $q>n^4$ is inaccessible with these tools.

The inequality concerning $N$ just obtained can be written in compact form $N=O(q^{d-1})$, since in general we are only interested in the order of magnitude of $N$.

Corollary \ref{subspace_of_invertibles} admits a more precise statement in the light of this inequality.

\begin{corollary} \label{irreducible_but_not_absolutely}
Let $\M$ be a subspace of $M_n(\mathbb{F}_q)$ of dimension $d$. Suppose that $n\geq 4$ and each element
of $\M^\times$ has non-zero determinant, and $2d>n$. Then the determinantal polynomial $P_{\M}$ is irreducible but if $q>n^6$, it is not absolutely irreducible.
\end{corollary}

\begin{proof}
Corollary \ref{subspace_of_invertibles} implies that $P_{\M}$ is irreducible. Now the number $N$ of zeros of $P_{\M}$ is 1 under the assumption that each element of $\M^\times$ has non-zero determinant. However, as $d>2$ under the given hypotheses, Theorem \ref{Lang_Weil}
implies that $N$ is at least $q^2/4$ and we have a contradiction. Thus $P_{\M}$ is not absolutely irreducible if $q$ is sufficiently large.
\end{proof}


\section{Determinantal polynomials that are not absolutely irreducible}

\noindent Let $\M$ be a subspace of $M_n(\mathbb{F}_q)$. We have shown in the previous section that there are certain sufficient conditions that ensure that $P_{\M}$ is irreducible. In this section we wish to explore consequences of the hypothesis that
$P_{\M}$ is irreducible but not absolutely irreducible. 

Suppose that $P_{\M}$ is irreducible but not absolutely irreducible. Then we may factor
$P_{\M}$ into absolutely irreducible factors in some extension field of finite degree over
$\mathbb{F}_q$. Let $g$ be an absolutely irreducible factor of degree $m$ of $P_{\M}$
with coefficients in $\mathbb{F}_{q^s}$, but not in any smaller subfield $\mathbb{F}_{q^t}$,
where $1\leq t<s$. As shown in \cite{GS}, Lemma 2, we have $n=sm$ and there is a factorization
\[
P_{\M}=\lambda gg^\sigma \cdots g^{\sigma^{s-1}},
\]
where $\lambda\in \mathbb{F}_q$ and $\sigma$ denotes the Frobenius mapping $a\mapsto a^q$
acting on the coefficients of $g$. Moreover, since we may express $\lambda$ as a product
\[
\lambda=\mu \mu^\sigma \cdots \mu^{\sigma^{s-1}},
\]
for some $\mu\in \mathbb{F}_{q^s}$, we may replace $g$ by $h=\mu g$ and obtain
\[
P_{\M}=hh^\sigma \cdots h^{\sigma^{s-1}},
\]
where $h$ is a homogeneous absolutely irreducible polynomial of degree $m$ in $d=\dim \M$ variables.

\begin{lemma} \label{characteristic_polynomial_from_determinantal_polynomial}
Let $\M$ be a $d$-dimensional subspace of $M_n(\mathbb{F}_q)$ whose determinantal
polynomial is irreducible but not absolutely irreducible. Suppose that the identity matrix
$I$ belongs to $\M$ and let $M_1=I$, $M_2$, \dots, $M_d$ be a basis of $\M$. Let
$P_{\M}$ be the determinantal polynomial evaluated with respect to this basis. Let
\[
P_{\M}=hh^\sigma \cdots h^{\sigma^{s-1}}
\]
be a factorization of $P_{\M}$ into $s$ Galois conjugate homogeneous polynomials 
of degree $m$ over $\mathbb{F}_{q^s}$, where $n=ms$.

Let $A=\lambda_1 I + \cdots +\lambda_d M_d$ be any element of $\M$, where the $\lambda_i$ are in $\mathbb{F}_q$. Then the characteristic polynomial $\det (yI-A)$ of $A$ is a product
of polynomials of the form 
\[
h(y-\lambda_1, -\lambda_2, \ldots,-\lambda_d)
\]
and its Galois conjugates over $\mathbb{F}_{q^s}$. Each of these polynomials has degree $m$.

\end{lemma}

\begin{proof}
We have 
\[
\det (x_1I+x_2M_2+\cdots +x_dM_d)=P_{\M}(x_1, x_2, \ldots, x_d).
\]
Thus if we replace $x_1$ by $y-\lambda_1$, $x_2$ by $-\lambda_2$, \dots, $x_d$ by $-\lambda_d$, we obtain
\[
\det (yI-A)=P_{\M}(y-\lambda_1, -\lambda_2, \ldots, -\lambda_d)
\]
and the rest follows from the factorization of $P_{\M}$.
\end{proof}

The splitting of the determinantal polynomial into Galois conjugate
factors is most exploitable when the degree of the polynomial is a prime, as we shall show
in the next few results.

\begin{theorem} \label{subspace_of_determinant_zero}
Let $r$ be a prime integer and let $\M$ be a subspace of $M_r(\mathbb{F}_q)$.  Suppose that the determinantal polynomial of $\M$ is irreducible but not absolutely irreducible. Then the elements of determinant zero in $\M$ are a subspace of $\M$. 
\end{theorem}

\begin{proof}
Let $d=\dim \M$ and let $M_1$, \dots, $M_d$ be a basis of $\M$. Let $P_{\M}$ be the determinantal polynomial of $\M$ evaluated with respect to this basis. Since we are assuming
that $P_{\M}$ is irreducible but not absolutely irreducible, the assumption that the polynomial has prime degree implies that it must factor over $\mathbb{F}_{q^r}$ into
$r$ Galois conjugate linear polynomials, say
\[
P_{\M}=\prod_{i=0}^{r-1}(\omega_1x_1 + \cdots +\omega_d x_d)^{\sigma^i}.
\]
Here, the $\omega_i$ are elements of $\mathbb{F}_{q^r}$ and $\sigma$ generates the Galois
group of $\mathbb{F}_{q^r}$ over $\mathbb{F}_{q}$.

Now an element $A=\lambda_1 M_1+\cdots +\lambda_d M_d$ of $\M$ has determinant zero if and only if 
\[
0=P_{\M}(\lambda_1, \ldots, \lambda_d)=\prod_{i=0}^{r-1}(\omega_1\lambda_1 + \cdots +\omega_d \lambda_d)^{\sigma^i}.
\]
It follows that $\det A=0$ if and only if 
\[
\omega_1\lambda_1 + \cdots +\omega_d \lambda_d=0.
\]
But the set of $d$-tuples $(\lambda_1, \ldots, \lambda_d)$ in $\mathbb{F}_q^d$ that satisfy
this linear equation over $\mathbb{F}_{q^r}$ is clearly a subspace and thus the elements
in $\M$ of determinant 0 are also a subspace of $\M$.
\end{proof}

\begin{lemma} \label{same_characteristic_polynomial}
Let $r$ be a prime integer and let $\M$ be a subspace of $M_r(\mathbb{F}_q)$ that contains
the identity element $I$.  Suppose that the determinantal polynomial of $\M$ is irreducible but not absolutely irreducible. Let $\M_0$ be the subspace of all elements of determinant zero, in accordance
with Theorem \ref{subspace_of_determinant_zero}, and let $\N$ be a complement of $\M_0$ in
$\M$, with $I\in \N$. Then if $A\in \N$, $B\in \M_0$, we have
\[
\det(yI-(A+B))=\det(yI-A).
\]
Thus $A$ and $A+B$ have the same characteristic polynomial.
\end{lemma}

\begin{proof}
Let $d=\dim \M$ and let $s=\dim \N$. Let $M_1$, \dots, $M_s$ be a basis of $\N$ and $
M_{s+1}$, \dots, $M_d$ be a basis of $\M_0$. We then evaluate $P_{\M}$ with respect to the basis $M_1$, \dots, $M_d$ of $\M$. Over $\mathbb{F}_{q^r}$, we have the factorization
\[  
P_{\M}=\prod_{i=0}^{r-1}(\omega_1x_1 + \cdots +\omega_d x_d)^{\sigma^i},
\]
where the $\omega_i$ are elements of $\mathbb{F}_{q^r}$. 

Now $P_{\M}(0, \ldots, 0, \lambda_{s+1}, \ldots, \lambda_d)=0$ for all $(\lambda_{s+1}, \ldots, \lambda_d)$ in $\mathbb{F}_q^{d-s}$, since all elements of $\M_0$ have determinant 0. It follows that $\omega_{s+1}=\cdots=\omega_d=0$. 

Let $A$ and $B$ be arbitrary elements of $\N$ and $\M_0$, respectively, with
\[
 A=\lambda_1 M_1+\cdots+\lambda_sM_s, \qquad B=\lambda_{s+1} M_{s+1}+\cdots+\lambda_d M_d.
 \]
 Then we have
 \[
 \det(yI-(A+B))=P_{\M}(y-\lambda_1, \ldots, -\lambda_d)
 \]
 by Lemma \ref{characteristic_polynomial_from_determinantal_polynomial}, and this polynomial
 clearly equals $\det(yI-A)$, since $\omega_{s+1}=\cdots=\omega_d=0$. 
\end{proof}

\begin{corollary} \label{nilpotent}
Assume the hypotheses of Lemma \ref{same_characteristic_polynomial}. Then all elements of the subspace $\M_0$ of $\M$ are nilpotent.
\end{corollary}

\begin{proof}
This follows from Lemma \ref{same_characteristic_polynomial}, since $\det (yI-B)=\det(yI)=y^r$ for all $B$ in $\M_0$.
\end{proof}

We now show that it is not necessary to assume that the subspace $\M$ contains the identity 
in order to obtain a conclusion similar to Corollary \ref{nilpotent}.

\begin{corollary} \label{nilpotence_in_a_subspace_of_determinant_zero}
Let $r$ be a prime integer and let $\M$ be a subspace of $M_r(\mathbb{F}_q)$. Suppose that the determinantal polynomial of $\M$ is irreducible but not absolutely irreducible. Let $A$,
$B$ be elements of $\M$ that satisfy $\det A\neq 0$, $\det B=0$. Then $A^{-1}B$ is nilpotent.
\end{corollary}

\begin{proof}
We form the subspace $\N=A^{-1}\M$, which contains $I$ and has the same dimension as $\M$. 
Furthermore $P_{\N}$ and $P_{\M}$ are scalar multiples of each other. The result follows
from Corollary \ref{nilpotent}.
\end{proof}

\begin{theorem} \label{irreducible_polynomial}
Let $r$ be a prime integer and let $\M$ be a subspace of $M_r(\mathbb{F}_q)$ with $2\dim \M>r$. Suppose that each element of $\M^\times$ has non-zero determinant and that the determinantal polynomial of $\M$ is not absolutely irreducible (a supposition guaranteed if
$q>r^6$). Let
$A$ and $B$ be linearly independent elements of $\M^\times$. Then the characteristic polynomial of $A^{-1}B$ is irreducible of degree $r$. Furthermore, for fixed $A$, there are at most $r$ elements $B$ in $\M$ for which $A^{-1}B$ has a given irreducible characteristic polynomial.
\end{theorem}

\begin{proof}
We may replace the subspace $\M$ by $A^{-1}\M$, which contains the identity. It therefore suffices to prove that if $\M$ contains the identity element and $C$ is an element of $\M$
that is not a scalar multiple of the identity, then the characteristic polynomial of $C$ is irreducible and there are at most $r$ elements of $\M$ with the same characteristic polynomial as $C$. 

Let $d=\dim \M$ and let $M_1=I$, \dots, $M_d$ be  a basis of $\M$. Note that $P_{\M}$ is irreducible, by Corollary \ref{contains_subspace_of_invertibles}. 
Our previous discussion shows that, over $\mathbb{F}_{q^r}$, we can factor
$P_{\M}$ as a product of $r$ Galois conjugate linear polynomials, as follows:
\[  
P_{\M}=\prod_{i=0}^{r-1}(\omega_1x_1 + \cdots +\omega_d x_d)^{\sigma^i},
\]
where the $\omega_i$ are elements of $\mathbb{F}_{q^r}$. We note that the $\omega_i$ are in fact linearly independent over $\mathbb{F}_q$, since a non-trivial dependence between them
implies that $P_{\M}$ has a non-trivial zero in $\mathbb{F}_q^d$, contrary to the fact that
all elements of $\M^\times$ have non-zero determinant. 

Let $C$ be written in terms of the basis as
\[
C=\lambda_1 I+\lambda_2M_2+\cdots +\lambda_d M_d,
\]
where at least one $\lambda_i$ is non-zero for $i\geq 2$. We set
\[
\alpha=\lambda_1 +\lambda_2 \omega_2 +\cdots +\lambda_d\omega_d.
\]
Then the linear independence of the $\omega_i$ implies that $\alpha$ is in $\mathbb{F}_{q^r}$ but not in $\mathbb{F}_q$. 

It follows from Lemma \ref{characteristic_polynomial_from_determinantal_polynomial} that 
the characteristic polynomial of $C$ is the product of the linear polynomial
$y-\alpha$ and its Galois conjugates. Since $\alpha$ is not in $\mathbb{F}_q$, it has $r$ different Galois conjugates and $\det(yI-C)$ is irreducible of degree $r$, as claimed.

Suppose next that $D$ is an element of $\M$ with the same characteristic polynomial as $C$.
Write
\[
D=\mu_1 I+\mu_2M_2+\cdots +\mu_d M_d,
\]
where the $\mu_i$ are in $\mathbb{F}_q$. Then since the eigenvalues of $C$ are $\alpha^{q^i}$, $0\leq i\leq r-1$, and the same is true of $D$, we must have
\[
\alpha^{q^i}=\mu_1 +\mu_2 \omega_2 +\cdots +\mu_d\omega_d
\]
for some $i$ satisfying $0\leq i\leq r-1$. 

Now if there are more than $r$ elements of $\M$ with the same characteristic polynomial,
there must exist $(\nu_1, \ldots, \nu_d)\neq (0, \ldots, 0)$ in $\mathbb{F}_q^d$ with
\[
0=\nu_1+\nu_2 \omega_2+\cdots +\nu_d \omega_d.
\]
This contradicts the linear independence of the $\omega_i$. Hence at most $r$ elements
of $\M$ have the given characteristic polynomial, as required.
\end{proof}

We require that $P_{\M}$ is irreducible but not absolutely irreducible for the proof of Theorem \ref{irreducible_polynomial} and have pointed out that this is automatic under the given hypotheses if $q$ is sufficiently large. When $q$ is small, behaviour with regard
to characteristic polynomials can be completely different from that described in the conclusions of the theorem. Thus, for example, there is a 5-dimensional subspace $\M$, say, of $M_5(\mathbb{F}_2)$ that contains the identity
and whose 31 non-zero elements all have determinant equal to 1. The 30 elements of $\M$ different from 0 and $I$ fall into two subsets each
of size 15. The elements in one subset have characteristic polynomial $(x^2+x+1)(x^3+x+1)$, those in the other have characteristic
polynomial $(x^2+x+1)(x^3+x^2+1)$. There are two equivalence classes of such subspaces.

Similarly, there is a 7-dimensional subspace $\N$, say, of $M_7(\mathbb{F}_2)$ that contains the identity and whose 127 non-zero elements
all have determinant 1. Some elements of $\N$ have irreducible characteristic polynomial but not every irreducible polynomial
of degree 7 occurs as a characteristic polynomial. Some elements of $\N$ different from 0 and $I$ have a reducible characteristic polynomial.
We are grateful to John Sheekey for providing information about these two unusual subspaces.

We proceed to show that the hypothesis $2\dim \M>r$ in Theorem \ref{irreducible_polynomial} is crucial  by means of a simple example.

\begin{theorem} \label{reducible_polynomials}
Let $d$ be a positive integer and set $n=2d+1$. Then there exists a subspace $\M$ of $M_n(\mathbb{F}_q)$ of dimension $d$ with the property
that each element of $\M^\times$ has non-zero determinant, yet for all elements $A$ and $B$ of $\M^\times$, the characteristic polynomial
of $A^{-1}B$ is reducible.
\end{theorem}

\begin{proof}
For each positive integer $t$, we may embed $\mathbb{F}_{q^t}$ into $M_t(\mathbb{F}_q)$. Thus, as is well known, $M_t(\mathbb{F}_q)$ contains
a $t$-dimensional subspace in which each non-zero element has non-zero determinant.

Let $\M_1$, $\N_1$ be such subspaces of $M_d(\mathbb{F}_q)$ and $M_{d+1}(\mathbb{F}_q)$ of dimensions $d$ and $d+1$, respectively. Let $M_1$,
\dots, $M_d$ be a basis of $\M_1$ and let $N_1$, \dots, $N_d$ be linearly independent elements in $\N_1$. Let $\M$ be the $d$-dimensional
subspace of $M_{2d+1}(\mathbb{F}_q)$ with basis consisting of the elements
\[
\left(
\begin{array}
{cc}
        M_i&0\\
        0&N_i  
\end{array}
\right),
\]
$1\leq i\leq d$. Clearly, the elements of $\M^\times$ have non-zero determinant but the characteristic polynomial of $A^{-1}B$ is reducible
for all pairs $A$ and $B$ in $\M^\times$.
\end{proof}

Let us now show that, in certain circumstances, the subspace constructed above is maximal with respect to the property that all its
non-zero elements have non-zero determinant.

\begin{corollary} \label{maximal_subspace}
Let $r=2d+1$ be a prime, with $d>1$, and let $\M$ be a subspace of dimension $d$ in $M_r(\mathbb{F}_q)$ of the kind constructed in Theorem \ref{reducible_polynomials}. Then provided $q$ is sufficiently large (say $q>r^6$), $\M$ is contained in no strictly larger subspace
$\M_1$, say, with the property that all elements of $\M_1$ have non-zero determinant.
\end{corollary}

\begin{proof}
Suppose that  $\M$ is contained in a strictly larger subspace
$\M_1$ with the property that all elements of $\M_1^\times$ have non-zero determinant. Then since $\dim \M_1\geq d+1$, it follows that
$2\dim \M_1>r$. In turn, Corollary \ref{contains_subspace_of_invertibles} implies that $P_{\M_1}$ is irreducible. On the other hand,
the Lang-Weil theorem implies that $P_{\M_1}$ is not absolutely irreducible if $q$ is sufficiently large compared with $r$.
Thus, assuming that $P_{\M_1}$ is not absolutely irreducible, Theorem \ref{irreducible_polynomial} implies that for all pairs $A$ and $B$ in $\M_1^\times$, with $A$ and $B$ linearly independent, the characteristic polynomial of $A^{-1}B$ is irreducible. However, as $d>1$, we may
choose $A$ and $B$ linearly independent in $\M$ and then $A^{-1}B$ has a reducible polynomial. This is a contradiction, and we see that
$\M$ is maximal with the non-vanishing determinant property.
\end{proof}

It is a matter of some importance to find  sufficient conditions for a determinantal polynomial to be irreducible. We give one such condition
in the case that we are dealing with a subspace of $r\times r$ matrices where $r$ is a prime. Our proof requires a subsidiary result,
which we present next as a separate lemma.

\begin{lemma} \label{codimension_n}
Let $K$ be a field and let $\M$ be a subspace of $M_n(K)$. Suppose that the elements 
of determinant zero in $\M$ are a subspace $\M_0$ of $\M$ of codimension $n$. Then $\M_0=0$
and all elements of $\M^\times$ have non-zero determinant.
\end{lemma}

\begin{proof}
We consider $\M$ acting by left multiplication on the vector space $V=K^n$ of column vectors
over $K$. Given a vector $v\in V$, we define $\epsilon_v: \M\to V$ by 
\[
\epsilon_v(M)=Mv
\]
for all $M\in\M$. The rank-nullity theorem implies that
\[
\dim \M=\ker \epsilon_v+\dim \epsilon_v(\M).
\]
Clearly, $\dim \epsilon_v(\M)\leq n$, and we see that the  codimension of $\ker \epsilon_v$ in $\M$ is at most $n$. Furthermore, any element of $\ker \epsilon_v$ has determinant 0, since it annihilates the vector $v$. Thus $\ker \epsilon_v\leq \M_0$, since
$\M_0$ consists of all elements of $\M$ of determinant 0. 

We deduce that $\ker \epsilon_v=\M_0$, since $\M_0$ has codimension $n$ in $\M$, whereas
$\ker \epsilon_v$ has codimension at most $n$. It follows that $\M_0$ annihilates all
elements of $V$ and hence is the zero subspace. 
\end{proof}

\begin{theorem} \label{dimension_r_absolutely_irreducible}
Let $r$ be a prime and let $\M$ be a $d$-dimensional subspace of $M_r(\mathbb{F}_q)$. Suppose that $\M$ contains an $r$-dimensional subspace $\N$
such that all elements of $\N^\times$ have non-zero determinant. Suppose also that $\dim \M>r$. Then $P_{\M}$ is absolutely irreducible
and if $q$ is sufficiently large, say $q>r^6$, the number of elements in $\M$ with determinant zero is $O(q^{d-1})$.
\end{theorem}

\begin{proof}
Corollary \ref{subspace_of_invertibles} implies that $P_{\M}$ is irreducible. If $P_{\M}$ is not absolutely irreducible, the elements
of determinant zero are a subspace, by Theorem \ref{subspace_of_determinant_zero}. This contradicts Lemma \ref{codimension_n}.
The estimate for the number of elements of determinant zero in $\M$ follows from the Lang-Weil theorem.
\end{proof}

Theorem \ref{dimension_r_absolutely_irreducible} depends crucially on the primality of $r$, as is obvious from its manner of proof.
We illustrate this point with an example obtained by field reduction.

\noindent{\bf Example 2 } Let $m>1$ and $s>1$ be integers. Consider the space $M_m(\mathbb{F}_{q^s})$ of $m\times m$ matrices over
$\mathbb{F}_{q^s}$. This space has dimension $m^2$ over $\mathbb{F}_{q^s}$. We may consider $M_m(\mathbb{F}_{q^s})$ as a subspace
of dimension $m^2s$ in $M_{ms}(\mathbb{F}_q)$. It is clear that $\M$ contains a subspace $\N$ of dimension $ms$ in which each element
of $\N^\times$ has non-zero determinant. However, it is easy to see that the number of elements of determinant zero in
$\M$ is $O(q^{s(m^2-1)})$, rather than $O(q^{sm^2-1})$, which an analogy with Theorem \ref{dimension_r_absolutely_irreducible} would suggest.
Of course, $P_{\M}$ is not absolutely irreducible, a fact that can be explained by the Lang-Weil theorem if $q$ is sufficiently large,
although simpler explanations can be given, valid for all $q$.

\section{Examples where the elements of determinant zero are a subspace}

\noindent We turn to considering examples of subspaces of matrices in which the non-invertible elements are a proper subspace of the given space.
We are able to work in greater generality than the domain of finite fields.

We acknowledge the help of John Sheekey in providing the ideas used to construct the subspaces in Theorems \ref{three_dimensional_subspace}
and \ref{larger_subspace_of_noninvertibles}.

Let $K$ be a field and $L$ a Galois extension field of degree 3 over $K$, with cyclic Galois
group generated by $\sigma$. We consider $L$ as a vector space of dimension 3 over $K$ and 
$\sigma$ as a $K$-linear endomorphism of $L$. Consider the set of all $K$-linear endomorphisms of $L$ of the form $T_{a,b}$ where
\[
T_{a,b}(z)=a(\sigma^2-\sigma)(z)+bz
\]
for all $z\in L$. Here, $a$ runs over the elements of $K$ and $b$ over the elements of trace 0 in $L$ (so that $b+\sigma(b)+\sigma^2(b)=0$). 

\begin{theorem} \label{three_dimensional_subspace}
The set of all $T_{a,b}$ described above is a three-dimensional subspace, $\N$, say, of
$K$-linear endomorphisms of $L$. The only elements of $\N$ that are not invertible are the
$K$-multiples of $\sigma^2-\sigma$. Thus the non-invertible elements of $\N$
are a one-dimensional subspace, whose non-zero elements have rank $2$.
\end{theorem}

\begin{proof}
It is clear that $\sigma^2-\sigma$ is a non-invertible element of $\N$, whose kernel is the
one-dimensional subspace of elements of $K$. Suppose now that $T_{a,b}$ is not invertible
and $z\neq 0$ is in the kernel of $T_{a,b}$. We wish to show that $b=0$ and thus
$T_{a,b}$ is a scalar multiple of $\sigma^2-\sigma$. For this purpose, it will suffice to assume that $a=1$.

We thus have 
\[
bz-\sigma(z)+\sigma^2(z)=0.
\]
We apply $\sigma$ twice to this equality and use the fact that $\sigma^3=1$. We obtain additionally
\[
z+\sigma(b)\sigma(z)-\sigma^2(z)=0
\]
and 
\[
-z+\sigma(z)+\sigma^2(b)\sigma^2(z)=0.
\]
This is a homogeneous system of linear equations in $z$, $\sigma(z)$ and $\sigma^2(z)$
whose coefficient matrix 
\[
\left(
\begin{array}
{ccc}
        b&-1&1\\
        1&\sigma(b)&-1\\
        -1&1&\sigma^2(b)   
\end{array}
\right), 
\]
must have zero determinant, since we are assuming that there is a non-trivial solution. 

We calculate that the determinant of the matrix above is
\[
b\sigma(b)\sigma^2(b)+b+\sigma(b)+\sigma^2(b).
\]
Thus, since $b+\sigma(b)+\sigma^2(b)=0$ by hypothesis, we must have $b\sigma(b)\sigma^2(b)=0$
and hence $b=0$, as required.
\end{proof}

We illustrate the ideas of the result above with a specific example in $M_3(\mathbb{F}_q)$.

\bigskip

\noindent{\bf Example 3 } Suppose that the prime power $q$ satisfies $q\equiv 1\bmod 3$. Let $b$ be an element of $\mathbb{F}_q$ that is not a cube. The polynomial $x^3-b$ is then irreducible in $\mathbb{F}_q[x]$. The elements 
\[
B=\left(
\begin{array}
{ccc}
        0&1&0\\
        0&0&1\\
        b&0&0   
\end{array}
\right), \qquad C=B^2=
\left(
\begin{array}
{ccc}
        0&0&1\\
        b&0&0\\
        0&b&0   
\end{array}
\right)
\]
span a two-dimensional subspace of $M_3(\mathbb{F}_q)$ in which every non-zero element 
has non-zero determinant. 

Let $\alpha$ be a root of $x^3-2$ in $\mathbb{F}_{q^3}$. Then $1$, $\alpha$ and $\alpha^2$ are a basis of $\mathbb{F}_{q^3}$ over $\mathbb{F}_{q}$ and we can identify $B$ with
$\alpha$ and $C$ with $\alpha^2$. Let $\sigma$ be the Frobenius mapping of $\mathbb{F}_{q^3}$ into itself and let $\tau=\sigma^2-\sigma$. Then we find that 
\[
\tau(1)=1, \qquad \tau(\alpha)=(\omega^2-\omega)\alpha,\qquad \tau(\alpha^2)=(\omega-\omega^2)\alpha^2,
\]
where $\omega$ is a primitive third root of unity in $\mathbb{F}_q$. The matrix of $\tau$ with respect to the chosen basis is thus
\[
A=\left(
\begin{array}
{ccr}
        0&0&0\\
        0&\lambda&1\\
        0&0& -\lambda
\end{array}
\right),
\]
where $\lambda=\omega^2-\omega$. 

Let $\M$ be the three-dimensional subspace of $M_3(\mathbb{F}_q)$ spanned by $A$, $B$ and $C$. We find that 
\[
P_{\M}=\det (x_1A+x_2B+x_3C)=b(x_2^3+bx_3^3).
\]
This polynomial is irreducible but factors into three conjugate linear polynomials over
$\mathbb{F}_{q^3}$. The elements of determinant 0 in $\M$ form a one-dimensional subspace spanned by $A$, as predicted by Theorem \ref{three_dimensional_subspace}.  

\begin{corollary} \label{larger_subspace_of_noninvertibles}
For each positive integer $m$, there is a subspace $\M$ of $M_{3m}(\mathbb{F}_q)$ of dimension $3m$ in which the elements of determinant
zero are a subspace of dimension $m$. The non-zero elements of this subspace have rank $2m$.
\end{corollary}

\begin{proof}
We work over the field $\mathbb{F}_{q^m}$ initially. Let $\N$ be a subspace of dimension three in $M_3(\mathbb{F}_{q^m})$ in which the elements
of determinant zero are a one-dimensional $\mathbb{F}_{q^m}$-subspace. Let $\M$ denote $\N$ considered as a vector space over
$\mathbb{F}_q$. Then $\dim \M=3m$ and we may consider $\M$ to be a subspace of $M_{3m}(\mathbb{F}_q)$. It is clear that
the one-dimensional $\mathbb{F}_{q^m}$-subspace of elements of zero determinant in $\N$ becomes an $m$-dimensional such subspace
in $\M$. Furthermore, since the non-zero elements of the subspace of non-invertibles have rank two over $\mathbb{F}_{q^m}$, they
have rank $2m$ over $\mathbb{F}_q$.
\end{proof}

We consider a further construction of a non-trivial subspace where the non-invertible elements are a subspace, this time consisting of
$4\times 4$ matrices in characteristic 2.

\begin{theorem} \label{example_in_four_dimensions}
Let $K$ be a field of characteristic $2$ and let $L$ be an extension field of degree four over $K$ with cyclic Galois group.
Then there exists a four-dimensional subspace $\M$ of $M_4(K)$ with the property that 
the elements of $\M$ with determinant zero are a one-dimensional subspace. $\M$ contains the identity, and the non-zero elements
of determinant zero have rank two and are nilpotent.
\end{theorem}

\begin{proof}
We consider $L$ as a four-dimensional vector space over $K$ and let $\sigma$ generate the Galois group of $L$ over $K$. We consider the $K$-linear
transformations $T_{a,b}: L\to L$ given by
\[
T_{a,b}(z)=a\sigma^2(z)+(a+c)z,
\]
where $a\in K$, $c\in L$, and $c$ has trace zero under the trace mapping from $L$ to $K$ (so that $c+\sigma(c)+\sigma^2(c)+\sigma^3(c)=0$).

It is clear that if $a\neq 0$, $T_{a,0}$ annihilates the fixed field of $\sigma^2$, which has degree 2 over $K$, and has rank two. It is also easy to verify that $T_{a,0}^2=0$, which tells us that $T_{a,0}$ is nilpotent.
Our aim now is to show that $T_{a,c}$ is invertible if $c\neq 0$. For this, it will suffice to show that  $T_{1,c}$ is invertible.

Let $z\in L$ be in the kernel of $T_{1,c}$. We have then
\[
\sigma^2(z)=(1+c)z.
\]
We replace $1+c$ by $d$, which also has trace zero, and then try to show that if $z\neq 0$ and $\sigma^2(z)=dz$ for some $d$ with trace zero,
then $d=1$. 

We have $d=\sigma^2(z)z^{-1}$ and hence $\sigma^2(d)d=1$. Thus $\sigma^2(d)=d^{-1}$ and $\sigma^3(d)=\sigma(d)^{-1}$. Since $d$ has trace zero, we obtain
\[
0=d+\sigma(d)+d^{-1}+\sigma(d)^{-1}=(d+\sigma(d))(1+d^{-1}\sigma(d)^{-1}).
\]
It follows that $d=\sigma(d)$ or $d^{-1}=\sigma(d)$. 

Suppose that $d=\sigma(d)$. Then $\sigma^2(d)=d$ and hence $\sigma^2(d)d=d^2$. But we know that $\sigma^2(d)d=1$ and hence $d=1$, since we are working
in characteristic two. On the other hand, if $d^{-1}=\sigma(d)$, then $\sigma^2(d)=d$ and hence, since $d\sigma^2(d)=1$, we see that $d^2=1$.
This again leads to the conclusion that $d=1$, and hence $c=0$, as required.
\end{proof}

We consider the special case that $K=\mathbb{F}_q$ and $L=\mathbb{F}_{q^4}$, with $q$ even.
Let $\M_0$ be the subspace of elements of determinant zero in $\M$ in the theorem above. Let $\N$ be a complement to $\M_0$ in $\M$.
Then each element of $\N^\times$ has non-zero determinant. Corollary \ref{contains_subspace_of_invertibles} implies that
$P_{\M}$ is irreducible. The result that the elements of $\M_0$ are nilpotent is in the spirit of Corollary \ref{nilpotence_in_a_subspace_of_determinant_zero}, although we have proved it in a dimension that is not a prime.

The following corollary of Theorem \ref{example_in_four_dimensions} follows from the method of field reduction, as practised
in Corollary \ref{larger_subspace_of_noninvertibles}.

\begin{corollary} \label{much_bigger_subspace}
Let $q$ be a power of $2$. Then for each positive integer $m$, there is a subspace $\M$ of $M_{4m}(\mathbb{F}_q)$ of dimension $4m$ in which the elements of determinant
zero are a subspace of dimension $m$. The non-zero elements of this subspace have rank $2m$ and are nilpotent.
\end{corollary}

\section{Group actions on subspaces of matrices, centralizers and normalizers}

\noindent The subspaces of matrices, especially those whose non-zero elements have non-zero determinant, have special properties
which that make them interesting subjects for various actions of the corresponding general linear group. 




\begin{lemma} \label{scalar_multiple_of_identity}
Let $r$ be a prime and let $\M$ be a subspace of $M_r(\mathbb{F}_q)$. Suppose that each element of $\M^\times$ has non-zero
determinant  and $r<2\dim \M<2r$. Let $A$ be an invertible element of $M_r(\mathbb{F}_q)$ that satisfies $A\M=\M$. Then if $q>r^6$,
$A$ is a scalar matrix.
\end{lemma}

\begin{proof}
Suppose that $A$ is not a scalar matrix. Let $B$ be an element of $\M^\times$. Then $B$ and $AB$ are linearly independent (for
otherwise $A$ is a scalar matrix). It follows from Theorem \ref{irreducible_polynomial} that $B^{-1}AB$ has an irreducible characteristic
polynomial. As $B^{-1}AB$ and $A$ have the same characteristic polynomial, we deduce that $A$ has an irreducible characteristic polynomial,
$f$, say, in $\mathbb{F}_q[x]$. 

Let $T$ be the linear transformation of $\M$ induced by multiplication by $A$. Since $f(A)=0$, we have $f(T)=0$. Thus the minimal
polynomial of $T$ divides $f$. But $f$ is irreducible and hence $T$ has irreducible minimal polynomial of degree $r$. This contradicts
our hypothesis that $\dim \M<r$. We deduce that $A$ is a scalar matrix.
\end{proof} 

There are various permutation actions of the general linear group $GL(n,\mathbb{F}_q)$ on the subspaces of $M_n(\mathbb{F}_q)$. One is by left multiplication (or by right multiplication), as considered above. Another is by conjugation, an action we will examine later in this section.
In these actions, the scalar matrices act trivially and thus we have actions by the projective general linear group $PGL(n,\mathbb{F}_q)$.
Lemma \ref{scalar_multiple_of_identity} then admits a simple interpretation in these terms.

\begin{corollary} \label{projective_group_action}
Let $r$ be a prime and let $\M$ be a subspace of $M_r(\mathbb{F}_q)$. Suppose that each element of $\M^\times$ has non-zero
determinant  and $r<2\dim \M<2r$. Then if $q>r^6$, the $PGL(r,\mathbb{F}_q)$-orbit containing $\M$ is regular (its size is 
$|PGL(r,\mathbb{F}_q)|$).
\end{corollary}

We turn to extending this result to subspaces of dimension $r$ with the non-vanishing determinant property. As we shall see, two types
of behaviour occur.

\begin{lemma} \label{not_necessarily_scalar_multiple_of_identity}
Let $r$ be a prime and let $\M$ be a subspace of $M_r(\mathbb{F}_q)$. Suppose that each element of $\M^\times$ has non-zero
determinant  and $\dim \M=r$. Let $A$ be an invertible element of $M_r(\mathbb{F}_q)$ that satisfies $A\M=\M$. Then if $q>r^6$,
either $A$ is a scalar matrix or $A$ has irreducible characteristic polynomial and $\M=\mathbb{F}_q(A)C$ for some invertible
element $C$. In this second case, $\M$ is equivalent to the field $\mathbb{F}_{q^r}$.
\end{lemma}

\begin{proof}
Suppose that $A$ is not a scalar matrix. Then the proof of Lemma \ref{scalar_multiple_of_identity} shows that $A$ has irreducible characteristic polynomial, $f$, say. As before, let $T$ be the linear transformation of $\M$ induced by multiplication by $A$. Then $T$ has minimal polynomial $f$. It follows that $\M$ is a cyclic $T$-module and each element of $\M$ is expressible as $g(T)C$ for some polynomial
polynomial $g$ in $\mathbb{F}_q[x]$, and fixed element $C$ in $\M$. This means that $\M=\mathbb{F}_q(A)C$  and $\M$ is equivalent
to the subspace $\mathbb{F}_q(A)$, consisting of polynomials in $A$. This subspace is isomorphic to the field $\mathbb{F}_{q^r}$.
\end{proof}

\begin{corollary} \label{modified_projective_group_action}
Let $r$ be a prime and let $\M$ be a subspace of $M_r(\mathbb{F}_q)$. Suppose that each element of $\M^\times$ has non-zero
determinant  and $\dim\M=r$. Then if $q>r^6$, the $PGL(r,\mathbb{F}_q)$-orbit containing $\M$ is either regular or its size is
$|GL(r,\mathbb{F}_q)|/(q^r-1)$.
\end{corollary}

We turn to an investigation of conjugation action. First we make a definition in the usual spirit of action by conjugation.

\begin{definition} Let $\M$ be a non-zero subspace of $M_n(\mathbb{F}_q)$. The centralizer of $\M$ in $GL(n,\mathbb{F}_q)$
is the subset of all $C$ in $GL(n,\mathbb{F}_q)$ satisfying $CB=BC$ for all $B\in\M$. We denote this subset
by $C(\M)$. The normalizer of $\M$ in $GL(n,\mathbb{F}_q)$
is the subset of all $A$ in $GL(n,\mathbb{F}_q)$ satisfying $A\M A^{-1}=\M$. We denote this subset
by $N(\M)$. 
\end{definition}

It is clear that $C(\M)$ and $N(\M)$ are subgroups of $GL(n,\mathbb{F}_q)$, and $C(\M)$ is a normal subgroup of $N(\M)$.

We shall make the assumption that our subspace $\M$ contains the identity
element. This is reasonable, as we need to know about the characteristic polynomials of elements of $\M$, which are of course invariant under conjugation, and Theorem \ref{irreducible_polynomial} provides us with exploitable information on this subject  if $\M$ contains
the identity.

We will confine our attention to subspaces $\M$ of $M_r(\mathbb{F}_q)$, where $r$ is a prime, all elements of $\M^\times$ have non-zero
determinant, $q$ is sufficiently large, and $2\dim \M>r$. In these circumstances, our first result shows that $C(\M)$ usually consists
of scalar matrices.

\begin{theorem} \label{centralizer_structure}
Let $r$ be a prime and let $\M$ be a subspace of $M_r(\mathbb{F}_q)$ that contains the identity matrix. Suppose that each element of $\M^\times$ has non-zero
determinant  and $2\dim\M>r$. Suppose also that $q>r^6$. Then either $C(\M)$ consists of scalar matrices or there is an element $C$ in $GL(r,\mathbb{F}_q)$
with irreducible characteristic polynomial such that $\M$ is a subspace of $\mathbb{F}_q(C)$, $C(\M)$ is the subgroup
 of invertible elements of $\mathbb{F}_q(C)$, and is cyclic of order $q^r-1$.

\end{theorem}

\begin{proof}
Let $C$ be an element of $C(\M)$ that is not a scalar matrix. Then $C$ commutes with a non-scalar element of $\M$, $B$, say. 
Since $\M$ contains the identity, Theorem \ref{irreducible_polynomial}  implies that $B$ has irreducible characteristic polynomial.
This in turn implies that $C$ is a polynomial in $B$, and since $C$ is not a scalar matrix, its characteristic polynomial
is also irreducible. Thus, since $C$ commutes with all elements of $\M$, they are all polynomials in $C$, and hence
$\M$ is a subspace of $\mathbb{F}_q(C)$. The rest follows from the theory of finite fields.
\end{proof}

Our final objective is to investigate the quotient group $N(\M)/C(\M)$ when $\M$ is as above. We will show the group is small, of order
no larger than $r$ but anticipate that better information should be available. To facilitate arguments, we will assume that
$r$ is odd and relatively prime to $q-1$. In these circumstances, we have the following elementary result, whose proof is straightforward.

\begin{lemma} \label{relatively_prime}
Let $r$ be an odd prime and suppose that $r$ is relatively prime to $q-1$. Then $q-1$ is relatively prime to $(q^r-1)/(q-1)$.
\end{lemma}

Let $\M$ be a non-zero subspace of $M_r(\mathbb{F}_q)$. We set $G=N(\M)/C(\M)$. There is a homomorphism $\theta$, say, from $N(\M)$ into
the group of $\mathbb{F}_q$-automorphisms of $\M$ given by 
\[
\theta(A)B=A^{-1}BA
\]
for $A\in N(\M)$ and $B\in \M$. The kernel is $C(\M)$ and thus $\theta(G)$ is isomorphic to $G$. We can thus say that $G$ acts faithfully
as a group of $\mathbb{F}_q$-linear automorphisms of $\M$. Using this notation, we have the following technical result.

\begin{theorem} \label{fixed-point_free_action}
Let $r$ be an odd prime that is relatively prime to $q-1$. Let $\M$ be a subspace of $M_r(\mathbb{F}_q)$
that contains the identity matrix. Suppose that each element of $\M^\times$ has non-zero
determinant  and $2\dim\M>r$. Suppose also that $q>r^6$. Then no non-identity
element of $\theta(G)$ fixes a non-scalar matrix in $\M$.
\end{theorem}

\begin{proof}
Suppose that $\theta(A)$ is non-trivial and $\theta(A)$ fixes a non-scalar matrix $B$, say, in $\M$. 
Then $A$ commutes with $B$. Now $B$ has an irreducible characteristic polynomial by Theorem \ref{irreducible_polynomial}. There are two cases
to consider. In the first, easier case, $\M$ consists of polynomials in $B$ and $A$ is itself a polynomial in $B$. This means
that $A$ centralizes $\M$ and contradicts the assumption that $\theta(A)$ is non-trivial. Thus, we may assume that $C(\M)$ consists
of scalar matrices, in accordance with Theorem \ref{centralizer_structure}.

Returning to consideration of $B$, since its characteristic polynomial is irreducible,
the centralizer, $W$, say, of $B$ in $GL(r,\mathbb{F}_q)$ is cyclic of order $q^r-1$. 
Now suppose that the order of $\theta(A)$ acting as an automorphism of $\M$ is $k$. Note that we are assuming that $k>1$. Then $k$ is the smallest positive integer such that $A^k$ is an element of $C(\M)$, and hence is a scalar matrix. Let $Z$ be the group of scalar matrices, which is of course contained in $W$.
It is easy to see that $k$ is the order of the coset $AZ$ in the group $W/Z$. Since $W/Z$ has order $(q^r-1)/(q-1)$, $k$ divides
$(q^r-1)/(q-1)$. 

Let $\ell$ be a prime divisor of $k$. Certainly, $\ell$ divides $q^r-1$, but we claim that $\ell$ does not divide $q-1$. This follows 
Lemma \ref{relatively_prime}, given our hypothesis $r$ is relatively prime to $q-1$. It follows that $r$ is the order of $q$ modulo $\ell$
and thus $r$ divides $\ell-1$. An immediate consequence is that $r<\ell$, which is the main tool used in the argument.

Let $\theta(D)$ be a power of $\theta(A)$ of order $\ell$. Since $\theta(D)$ does not act trivially on $\M$, it has an orbit on
the elements of $\M$ of length $\ell$. Now if $B'$ is in such an orbit, it has an irreducible characteristic polynomial and all $\ell$
elements in the orbit have the same characteristic polynomial. But by Theorem \ref{irreducible_polynomial}, there are at most $r$ elements
in $\M$ that have a given irreducible characteristic polynomial. Since $r<\ell$, we have obtained a contradiction. We deduce that
$\theta(A)$ is trivial, as required.
\end{proof}

\begin{corollary} \label{order_is_r}
Let $r$ be an odd prime that is relatively prime to $q-1$. Let $\M$ be a subspace of $M_r(\mathbb{F}_q)$
that contains the identity matrix. Suppose that each element of $\M^\times$ has non-zero
determinant  and $\dim\M=r$. Suppose also that $q>r^6$. Then either  $N(\M)=C(\M)$ or $N(\M)/C(\M)$ has order $r$.
\end{corollary}

\begin{proof}
We have shown in Theorem \ref{irreducible_polynomial} that for any element of $\M$ that is not a scalar matrix, there are at most $r$
elements of $\M$ with the same irreducible characteristic polynomial. Since it is well known that there are exactly $(q^r-q)/r$
irreducible monic polynomials of degree $r$ over $\mathbb{F}_q$, it follows that given any non-scalar element of $\M$, there are exactly
$r$ elements of $\M$ that have the same characteristic polynomial (and all irreducible monic irreducible polynomials of degree $r$ occur).

Suppose that $G$ is non-trivial. Theorem \ref{fixed-point_free_action} implies that each  orbit of the action of $\theta(G)$ on the non-scalar
elements of $\M$ has size $|G|$. The elements in each orbit have the same characteristic polynomial. Suppose that exactly $t$ orbits
of this type consist of elements with the same characteristic polynomial. Then we must have $t|G|=r$. Since $r$ is a prime,
we conclude that $t=1$ and $|G|=r$. (Note then that the $(q^r-q)/r$ non-trivial orbits correspond to irreducible monic polynomials of degree $r$.)
\end{proof}

We would like to extend Corollary \ref{order_is_r} to subspaces of dimension less than $r$, but there is an obstacle in the way of a proof,
as we now explain.

\begin{corollary} \label{order_is_at_most_r}
Let $r$ be an odd prime that is relatively prime to $q-1$. Let $\M$ be a subspace of $M_r(\mathbb{F}_q)$
that contains the identity matrix. Suppose that each element of $\M^\times$ has non-zero
determinant  and $r<2\dim\M<2r$. Suppose also that $q>r^6$. Then  $G=N(\M)/C(\M)$ has order at most $r$. If $\M$ contains
$r$ non-scalar matrices with the same characteristic polynomial, $|G|$ equals 1 or $r$. If $\M$ does not contain
$r$ such matrices, $|G|$ is less than $r$.
\end{corollary}

\begin{proof}
The proof of the previous corollary shows that a non-trivial $\theta(G)$-orbit has size $|G|$. Since such orbits consist of elements
of the same characteristic polynomial, Theorem \ref{irreducible_polynomial} implies that $|G|\leq r$. If there is an irreducible monic
polynomial of degree $r$ that is the characteristic polynomial of $r$ elements of $\M$, then the argument of the previous
proof implies that if $|G|>1$, then $|G|=r$. Otherwise, $|G|<r$.
\end{proof}

The question of whether we can have $G$ non-trivial but of order less than $r$ has not been resolved. It is easy to see that $|G|$ is odd,
and if $G$ is abelian, it is cyclic.

\section{Subspaces with square determinants}

\noindent Let $K$ be a field and let $K[x_1, \ldots, x_d]$ be the polynomial ring in the $d$ independent indeterminates $x_1$, \dots, $x_d$ over $K$. Let $y$ be a further indeterminate independent of the $x_i$. 

\begin{lemma} \label{irreducible_if_not_a_square}
Let $f$ be a non-constant polynomial in $R=K[x_1, \ldots, x_d]$. Then $y^2-f$ is reducible in $R[y]$ if and only if $f=g^2$ for some
$g$ in $R$.
\end{lemma}

\begin{proof}
Clearly, if $f=g^2$ for some $g$ in $R$, we can factor $y^2-f$ as $(y-g)(y+g)$. Conversely, suppose that
\[
y^2-f=FG,
\]
where $F$ and $G$ are non-constant polynomials in $R[y]$. Working in $R[y]$, this factorization is only possible if $F$ and $G$ are polynomials of degree one in $R[y]$, say
\[
F=ay+F_1,\quad G=by+G_1,
\]
where $a$ and $b$ are elements of $K$, and $F_1$ and $G_1$ are both in $R$.

Comparing terms in the equality $y^2-f=FG$, we obtain
\[
ab=1,\quad aG_1+bF_1=0,\quad F_1G_1=-f.
\]
We deduce that $G_1=-b^2F_1$ and thus $f=b^2F_1^2$ is a square in $R$, as required.

\end{proof}

We specialize to the case that $K=\mathbb{F}_q$, where $q$ is a power of an odd prime, and investigate if the polynomial
$y^2-f$ considered above is absolutely irreducible.

\begin{theorem} \label{absolutely_irreucible}
Let $f$ be a polynomial in $\mathbb{F}_q[x_1, \ldots, x_d]$. Suppose that 
$y^2-f$ is reducible in 
$\overline{\mathbb{F}}_q[x_1, \ldots, x_d,y]$, where $\overline{\mathbb{F}}_q$ denotes the algebraic closure of $\mathbb{F}_q$. Then
$f=g^2$ for some polynomial $g\in E[x_1, \ldots, x_d]$, where $E=\mathbb{F}_q$ or $\mathbb{F}_{q^2}$. In the first case, $f$ is the square
of a polynomial in $\mathbb{F}_q[x_1, \ldots, x_d]$. In the second, $g^\sigma=-g$, where $\sigma$ is the Frobenius $q$-th power mapping applied
in $\mathbb{F}_{q^2}$.
\end{theorem}

\begin{proof}
Suppose that $y^2-f$ is reducible in $\overline{\mathbb{F}}_q[x_1, \ldots, x_d,y]$. Then by Lemma \ref{irreducible_if_not_a_square}, $f=g^2$
for some $g$ in $\overline{\mathbb{F}}_q[x_1, \ldots, x_d]$. Now the coefficients of $g$ lie in the algebraic closure and hence are all algebraic over $\mathbb{F}_q$. The coefficients therefore lie in some finite field $E$ that is an extension of $\mathbb{F}_q$, and thus
$g\in E[x_1, \ldots, x_d]$. The case where $E=\mathbb{F}_q$ is trivial, so we will assume that $E>\mathbb{F}_q$.

Let $\sigma$ be an element of the Galois group of $E$ over $\mathbb{F}_q$ that does not fix some coefficient of $g$. Given $h\in 
E[x_1, \ldots, x_d]$, let $h^\sigma$ denote the polynomial obtained from $h$ by applying $\sigma$ to the coefficients but not to the variables. Then we have
\[
f=f^\sigma=g^2=(g^\sigma)^2.
\]
It follows that 
\[
g^2-(g^\sigma)^2=0=(g-g^\sigma)(g+g^\sigma).
\]
Thus since $g$ and $g^\sigma$ belong to an integral domain, either $g=g^\sigma$ or $g=-g^\sigma$, and the first possibility is already excluded. We deduce that $g^\sigma=-g$ and hence $g^{\sigma^2}=g$. Thus the coefficients of $g$ lie in $\mathbb{F}_{q^2}$ and $g$
is an element of $\mathbb{F}_{q^2}[x_1, \ldots, x_d]$.

\end{proof}

Our intention is to apply this theorem, together with the Lang-Weil theorem, to show that polynomials are squares if they only take square values on evaluation. 

\begin{theorem} \label{square_polynomial}
Let $q$ be a power of an odd prime. Let $f\in \mathbb{F}_q[x_1, \ldots, x_d]$ be a non-zero polynomial of even degree $n$. Suppose that when
evaluated on $\mathbb{F}_q^d$, $f$ takes only square values (possibly zero).
Then provided $q>n^6$, $f$ is the square of a 
 polynomial in $\mathbb{F}_q[x_1, \ldots, x_d]$.
\end{theorem}

\begin{proof}

Let us first note that as $f$ is non-zero, it certainly takes non-zero values when evaluated on $\mathbb{F}_q^d$. This follows from Ore's
theorem, \cite{LN}, Theorem 6.13, given our hypothesis that $q>n^6$ ($q>n$ will suffice for this).
Now we consider the consequences of assuming that
the polynomial $F=y^2-f$ in $\mathbb{F}_q[x_1, \ldots,x_d,y]$ is absolutely irreducible.
 Suppose 
$f$ has exactly $r$ zeros in $\mathbb{F}_q^d$ and $F$ has $N$ zeros in
$\mathbb{F}_q^{d+1}$. Then we have  $N=r+2(q^d-r)=2q^d-r$, given the hypothesis that $f$ takes only square values. Now by the same theorem of Ore, 
since $f$ is non-zero, $r\leq nq^{d-1}$. Thus we have
\[
N\geq 2q^d-nq^{d-1}.
\]
It follows from the effective form of the Lang-Weil theorem due to Cafure and Matera
that
\[
N\leq q^d+(n-1)(n-2)q^{d-1/2}+5n^{13/3}q^{d-1}.
\]
This leads to the inequality
\[
2q^d-nq^{d-1}\leq q^d+(n-1)(n-2)q^{d-1/2}+5n^{13/3}q^{d-1}
\]
and when we divide by $q^{d-1/2}$, we obtain
\[
q^{1/2}\leq nq^{-1/2}+(n-1)(n-2)+5n^{13/3}q^{-1/2}.
\]

We now apply the hypothesis that $q\geq n^6$. The inequality becomes
\[
n^3\leq q^{1/2}\leq 1+(n-1)(n-2)+5n^{4/3}\leq 6n^2.
\]
This is certainly a contradiction if $n>6$. We can now argue a little more precisely when
$n=6$ or $n=4$. 

When $n=6$, our inequality is
\[
6^3=216\leq 1 +15+5\times 6^{4/3}
\]
and since $6^{4/3}<12$, we have another contradiction. The same explicit type of manipulations also eliminate the remaining possibilities that $n=4$ or $n=2$. 

Thus $F=y^2-f$ is not absolutely irreducible. As we have seen, this can happen in two ways. Either $f$ is the square of a polynomial
in $\mathbb{F}_q[x_1, \ldots, x_d]$, which is what we want to prove, or $f=g^2$, where $g$ is a polynomial
in $\mathbb{F}_{q^2}[x_1, \ldots, x_d]$, and $g^\sigma=-g$ for the Frobenius automorphism $\sigma$ of
$\mathbb{F}_{q^2}$. Consider the second possibility. Then as $f=g^2$ and $f$ takes only square values on $\mathbb{F}_q^d$, $g$ takes values
in $\mathbb{F}_q$ when evaluated on $\mathbb{F}_q^d$. This is incompatible with $g^\sigma=-g$, given that $g$ is not identically
zero on $\mathbb{F}_q^d$ (since $f$ is not identically zero). Thus $f$ is a square.
\end{proof}

We should remark that results like Theorem \ref{square_polynomial} do require some hypothesis on the size of $q$ compared with the degree of the polynomial. Let us illustrate this phenomenon with an example. Let $p$ be an odd prime. It is well known that
the Artin-Schreier polynomial $f=x^p-x+\lambda$, where $\lambda$ is a non-zero element of
$\mathbb{F}_p$ is irreducible in $\mathbb{F}_p[x]$. Now, provided $p\equiv 3 \bmod 4$,
the polynomial $x^{2p}-x^2+\lambda$ is also irreducible if and only if $\lambda$ is a square
(when $p\equiv 1 \bmod 4$,  $x^{2p}-x^2+\lambda$ is irreducible if and only if $\lambda$ is a non-square). Thus, the polynomial $x^{2p}-x^2 +1$ is irreducible provided $p\equiv 3 \bmod 4$. 

We now homogenize the polynomial to form
\[
F(x,y)=x^{2p}-x^2y^{2p-2}+y^{2p}.
\] 
We claim that $F(\lambda, \mu)$ is a square for all $(\lambda, \mu)\in \mathbb{F}_p$.
This is clear if $\lambda=0$ or $\mu=0$. Suppose then that both $\lambda$ and $\mu$ are non-zero. Then $\lambda^{2p}=\lambda^2$, $\mu^{2p-2}=1$ and $\mu^{2p}=\mu^2$. Thus
\[
F(\lambda, \mu)=\mu^2
\]
is a square, as claimed. Since $F$ is irreducible when $p\equiv 3 \bmod 4$, it is certainly
not the square of a polynomial.

\begin{corollary} \label{subspace_of_squares}
Let $\M$ be a $d$-dimensional subspace of $n\times n$ matrices over $\mathbb{F}_q$
where $q$ is odd and $d>0$. Suppose that the determinant of each element of $\M^\times$ is a non-zero square in $\mathbb{F}_q$. Then $n$ is even, say $n=2m$, and if $q\geq n^6$, we have
$d\leq m$. 
\end{corollary}

\begin{proof}
We first show that $n$ is even. Let $P_{\M}$ be the determinantal polynomial of $\M$.
Let $(\lambda_1, \ldots, \lambda_d)$ be a non-zero element
of $\mathbb{F}_q^d$. Then $P_{\M}(\lambda_1,\ldots, \lambda_d)$ is a non-zero square. Now let
$\lambda$ be a non-square in $\mathbb{F}_q$. Since $P_{\M}$ is defined as a determinant, it is homogeneous and hence
\[
P_{\M}(\lambda\lambda_1, \ldots, \lambda\lambda_d)=\lambda^nP_{\M}(\lambda_1, \ldots, \lambda_d).
\]
Since $P_{\M}(\lambda\lambda_1, \ldots, \lambda\lambda_d)$ is also a non-zero square, it follows
that $\lambda^n$ is also a square. This implies that $n$ is even, since $\lambda^n$ is a non-square if $n$ is odd. 

Now $P_{\M}$ is a homogeneous polynomial of degree $n$ with no non-trivial zeros, and its values are squares. It follows 
from Theorem \ref{square_polynomial} that if $q\geq n^6$, $P_{\M}=g^2$ for some polynomial $g$.
Suppose if possible that $d>m$. Then since $g$ has degree $m$ and $g(0, \ldots, 0)=0$,
$g$ has a non-trivial zero, by Chevalley's theorem, \cite{LN}, Theorem 6.6. But this implies that $f$ also has a non-trivial zero and this is not true. Thus, $d\leq m$.
\end{proof}

\begin{corollary} \label{characteristic_polynomial_a_square}
Let $\M$ be a subspace of $M_{2m}(\mathbb{F}_q)$ with the property that $\det M$ is a square (possibly zero) for all elements $M$ in
$\M$. Suppose that $\M$ contains the identity matrix. Then if $q>64m^6$, the characteristic polynomial of each element of $\M$ is the square
of a monic polynomial in $\mathbb{F}_q[x]$.
\end{corollary}

\begin{proof}
Let $d=\dim \M$ and let $M_1=I$, \dots, $M_d$ be a basis of $\M$. We then calculate the determinantal polynomial of $\M$ with respect to this basis.
We have in this case
\[
\det (x_1I+x_2M_2+\cdots +x_dM_d)=P_{\M}(x_1, x_2, \ldots, x_d).
\]
Let $M=\lambda_1 I+\cdots +\lambda_dM_d$. 
Thus if we replace $x_1$ by $x-\lambda_1$, $x_2$ by $-\lambda_2$, \dots, $x_d$ by $-\lambda_d$, we obtain
\[
\det (xI-M)=P_{\M}(x-\lambda_1, -\lambda_2, \ldots, -\lambda_d).
\]
Since $\det \M$ is a square for all $M$ in $\M$, and we are assuming that $q>64m^6$, Theorem \ref{square_polynomial} implies that
$P_{\M}=Q^2$ for some polynomial in $\mathbb{F}_q[x_1, \ldots, x_d]$. Then we obtain
\[
\det (xI-M)=Q(x-\lambda_1, -\lambda_2, \ldots, -\lambda_d)^2,
\]
as required.
\end{proof}

The following observation extends this result to subspaces that have the square determinant property but do not necessarily contain the identity matrix.

\begin{corollary} \label{no_identity_but_subspace_of_squares}
Let $\M$ be a subspace of $M_{2m}(\mathbb{F}_q)$ with the property that $\det M$ is a square (possibly zero) for all elements $M$ in
$\M$. Suppose also that $\M$ contains elements with non-zero determinant. Then if $q>64m^6$, given elements $M$ and $N$ in $\M$, with $\det M\neq 0$,
the characteristic polynomial of $M^{-1}N$ is the square
of a monic polynomial in $\mathbb{F}_q[x]$.
\end{corollary}

\begin{proof}
Consider the subspace $\N=M^{-1}\M$. It contains the identity element and all its elements have square determinant, by the multiplicative
property of determinants. This corollary then follows from Corollary \ref{characteristic_polynomial_a_square}.
\end{proof}

It is well known that over any field $K$, the determinant of 
a $2m\times 2m$ skew-symmetric matrix with entries in $K$ is a square in $K$. The determinantal polynomial of the space of skew-symmetric
matrices is the square of a polynomial, known as the Pfaffian. The space of $2m\times 2m$ skew-symmetric does not contain the identity matrix
and it is not necessarily true that the characteristic polynomial of a skew-symmetric is the square of a polynomial. It is true
that if $M$ and $N$ are $2m\times 2m$ skew-symmetric matrices over $K$, with $\det M\neq 0$, then the characteristic polynomial of
 $M^{-1}N$ is a square of a polynomial in $K$. This can be proved by using the Pfaffian, along the lines of Corollary \ref{no_identity_but_subspace_of_squares}, but it can also be proved by elementary linear algebra, without invoking the Pfaffian.
 
 \section{Estimates related to rank in subspaces of matrices}

\noindent The determinantal polynomial that we have investigated in the previous sections can really only have any use if we are examining
subspaces of square matrices in which there are elements of non-zero determinant. Nonetheless, it is reasonable to try to investigate
such properties as the number of elements of maximal rank in an arbitrary subspace and in this case, polynomials related to minors
can fill the gap caused by the lack of useful determinantal polynomial, as we shall explain in this section.

Given positive integers $m$ and $n$, with $m\leq n$, let $M_{m\times n}(\mathbb{F}_q)$
denote the space of $m\times n$ matrices with entries in $\mathbb{F}_q$. When $m=n$, we will continue to use the notation $M_{n}(\mathbb{F}_q)$ in place of $M_{n\times n}(\mathbb{F}_q)$.
While the rank of a given matrix may be found practically by row operations, for theoretical purposes the calculation of minors is  more important.

Suppose that $\M$ is a subspace of $M_{m\times n}(\mathbb{F}_q)$ and $r\leq m$ is the maximum of $\rank A$, as $A$  ranges over the elements of $\M$. If $q$ is large enough compared with $r$, we would expect almost all elements of $\M$ to have rank $r$. We make this statement precise in our next two result, which are presumably well known in some form, 
although we do not know a specific reference. We begin with $n\times n$ matrices of rank $n$.

\begin{theorem} \label{number_of_elements_of_given_rank_square_case}
Let $\M$ be a subspace of $M_{n}(\mathbb{F}_q)$ of dimension $d$. Suppose that $\M$ contains an element of rank $n$. Then if $\M_n$ denotes the subset of elements of $\M$ of rank $n$ (equivalently, the elements of non-zero determinant),
\[
|\M_n|\geq q^d-nq^{d-1}+n-1.
\]
\end{theorem}

\begin{proof}

Let $M_1$, \dots, $M_d$ be a basis of $\M$ and let the determinantal polynomial be computed with respect to this basis.
$\M_n$ is equal to the number of non-zeros of $P_{\M}$.
Now a theorem of Ore, \cite{LN}, 6.15, implies that the number of zeros of $P_{\M}$ in $\mathbb{F}_q^d$ is at most $n(q^{d-1}-1)+1$, and thus the lower bound for $|\M_n|$ follows.
\end{proof}

Note that when $d=1$, we have $|\M_n|=q-1$, and the inequality is replaced by
 a trivial equality.

This estimate has no content if $q\leq n$ and we cannot expect useful inequalities
of this kind when we work over small fields. When $q\geq 2n$, we can certainly assert
that $|\M_n|\geq |\M|/2$, and thus most elements of $\M$ are invertible in this case.

If we take $\M$  to be the subspace of all diagonal matrices in $M_{n}(\mathbb{F}_q)$,
the number of elements of rank $n$ in $\M$ is $(q-1)^n$. Theorem \ref{number_of_elements_of_given_rank_square_case} gives the estimate
 that $|\M_n|$ is at least $q^n-n(q^{n-1}-1)-1$, which  seems to be a reasonable approximation for large $q$. In the worst case, when $q=2$, $\M$ contains exactly one element of rank $n$, and thus the proportion of elements of rank $n$ in this subspace is vanishingly small as $n$ becomes large. 

Serre, \cite{Se}, has improved Ore's upper bound, \cite{LN}, Theorem 6.15, for the number of zeros in the homogeneous case. This increases our lower bound for the number of elements of rank  $n$ by a term $nq^{d-2}$. 

\begin{theorem} \label{number_of_elements_of_given_rank_rectangular_case}
Let $\M$ be a subspace of $M_{m\times n}(\mathbb{F}_q)$ of dimension $d$. Suppose that $\M$ contains an element of rank $r$. Then if $\M_r$ denotes the subset of elements of $\M$ of rank 
at least $r$, we have
\[
|\M_r|\geq q^d-rq^{d-1}+r-1.
\]
\end{theorem}

\begin{proof}
Let $S$ be an element of rank $r$ in $\M$. We know then from the elementary theory of matrices that there exist invertible $m\times m$ and $n\times n$ matrices, $A$, $B$, respectively, such that
\[
ASB=\left(
\begin{array}
{cc}
        I_r&0\\
        0&0   
\end{array}
\right),
\]
where $I_r$ denotes the $r\times r$ identity matrix. 

We may replace the given subspace $\M$ by the subspace $\N=A\M B$, if necessary, and then we may assume that 
\[
S=\left(
\begin{array}
{cc}
        I_r&0\\
        0&0   
\end{array}
\right)
\]
is in $\N$. Since the one-to-one linear transformation $Z\to AZB$ from $\M$ into $\N$
preserves rank, we do not change any rank properties in exchanging $\M$ for $\N$. 

Given any matrix $Z$ in $\N$, let $Z^*$ denote the $r\times r$ matrix obtained from $Z$ by taking its first $r$ rows and columns and let $\theta:\M\to M_{r}(\mathbb{F}_q)$ be the linear transformation given by $\theta(Z)=Z^*$. Let $e=\dim \theta(\M)$. We note that
$\theta(S)=I_r\in \theta(\M)$. Thus, if we let $\A$ denote the subset of elements of 
rank $r$ in $\theta(\M)$, we have
\[
|\A|\geq q^e-rq^{e-1}+r-1
\]
by Theorem \ref{number_of_elements_of_given_rank_square_case}. 

Finally, any element $Z$ in $\M$ such that $\theta(Z)\in\A$ is in $\M_r$, since
$\det Z^*$ is then a non-zero minor of order $r$. It follows that
\[
|\M_r|\geq q^{d-e}|\A|=q^d-rq^{d-1}+(r-1)q^{d-e}\geq q^d-rq^{d-1}+r-1,
\]
as required.
\end{proof}

The proof shows that our lower bound for $|\M_r|$ can be improved if we know that
$d>e$. Since $e\leq r^2$, this is guaranteed if $d>r^2$.

There is a version of Theorem \ref{number_of_elements_of_given_rank_square_case} that is valid for skew-symmetric matrices and gives improved bounds. In this connection, we recall the convention that in characteristic two, a skew-symmetric matrix is a symmetric matrix whose diagonal entries are all 0.
We also recall that the rank of a skew-symmetric matrix is necessarily even.

\begin{theorem} \label{number_of_invertible_skew_symmetric_elements}
Let $\M$ be a subspace of $M_{2m}(\mathbb{F}_q)$ of dimension $d$. Suppose that $\M$ 
consists of skew-symmetric matrices and contains an element of rank $2m$. Then if $\M_{2m}$ denotes the subset of elements of $\M$ of rank $2m$, we have
\[
|\M_{2m}|\geq q^d-mq^{d-1}+m-1.
\]
\end{theorem}

\begin{proof}
The proof is almost identical with that of Theorem \ref{number_of_elements_of_given_rank_square_case}.
Let $S_1$, \dots, $S_d$ be a basis of $\M$ and let $x_1$, \dots, $x_d$ be independent indeterminates over $\mathbb{F}_q$. We set
\[
f(x_1,\ldots,x_d)=\det(x_1S_1+\cdots +x_dS_d).
\]
Since we are dealing with skew-symmetric matrices, the theory of the Pfaffian shows that $f$ is the square of a non-zero homogeneous polynomial, $g$, say, of degree $m$ in $d$ indeterminates. 
The theorem of Ore, already quoted, implies that the number of zeros of $g$ in $\mathbb{F}_q^d$ is at most $m(q^{d-1}-1)+1$, and thus the lower bound for $|\M_{2m}|$ follows.
\end{proof}

Extension of this theorem to deal with the number of elements of rank at least some specified
even integer in a subspace of skew-symmetric matrices is straightforward, as we now show.

\begin{theorem} \label{number_of_skew_symmetric_elements_of_given_rank_or_more}
Let $\M$ be a $d$-dimensional subspace of $n\times n$ skew-symmetric matrices over $\mathbb{F}_q$. Suppose that $\M$ contains an element of rank $2r$. Then if $\M_{2r}$ denotes the subset of elements of $\M$ of rank 
at least $2r$, we have
\[
|\M_{2r}|\geq q^d-rq^{d-1}+r-1.
\]
\end{theorem}

\begin{proof}
Let $S$ be an element of rank $2r$ in $\M$. We know then from the elementary theory of 
skew-symmetric matrices that there exists an invertible  $n\times n$ matrix, $X$, say, such that
\[
XSX^T=\left(
\begin{array}
{cc}
        Z&0\\
        0&0   
\end{array}
\right),
\]
where $Z$ denotes an invertible  $2r\times 2r$ skew-symmetric matrix, and $X^T$ denotes the transpose of $X$.

We may replace the given subspace $\M$ by the subspace $\N=X\M X^T$, if necessary, and then we may assume that 
\[
S=\left(
\begin{array}
{cc}
        Z&0\\
        0&0   
\end{array}
\right)
\]
is in $\N$.  Again, we do not change any rank properties in exchanging $\M$ for $\N$. 

Given any matrix $A$ in $\N$, let $A^*$ denote the $2r\times 2r$ matrix obtained from $A$ by taking its first $2r$ rows and columns and let $\theta:\M\to M_{2r}(\mathbb{F}_q)$ be the linear transformation given by $\theta(A)=A^*$. Let $e=\dim \theta(\M)$. We note that $\theta(\M)$ is a subspace of skew-symmetric matrices and 
$\theta(S)$ is an invertible element in $\theta(\M)$. Thus, if we let $\A$ denote the subset of elements of 
rank $2r$ in $\theta(\M)$, we have
\[
|\A|\geq q^e-rq^{e-1}+r-1
\]
by Theorem \ref{number_of_invertible_skew_symmetric_elements}.  

Finally, any element $B$ in $\M$ such that $\theta(B)\in\A$ is in $\M_{2r}$, since
$\det B^*$ is then a non-zero minor of order $2r$. It follows that
\[
|\M_{2r}|\geq q^{d-e}|\A|=q^d-rq^{d-1}+(r-1)q^{d-e}\geq q^d-rq^{d-1}+r-1,
\]
as required.
\end{proof}

\end{document}